\documentclass[11pt]{amsart}

\usepackage{amsmath}
\usepackage{amssymb, amscd, fixmath}
\usepackage{amsthm}
\usepackage{enumerate, comment}
\usepackage[cmtip,all]{xy}

\theoremstyle{plain}
\newtheorem{thm}{Theorem}[section]
\newtheorem{lem}[thm]{Lemma}
\newtheorem{prop}[thm]{Proposition}
\newtheorem{cor}[thm]{Corollary}

\theoremstyle{remark}

\newtheorem*{lem_geometric}{\bf Lemma  \ref{L:geometric}}

\numberwithin{equation}{section}

\def\Hom{\operatorname{Hom}}
\def\Ind{\operatorname{Ind}}
\def\Irr{\operatorname{Irr}}
\def\Isom{\operatorname{Isom}}

\def\GL{\mathrm{GL}}
\def\Mp{\mathrm{Mp}}

\def\CC{\mathbb{C}}

\begin{document}

\title[The Howe duality conjecture]
{A proof of the Howe duality conjecture}
\author{Wee Teck Gan}
\address{Department of Mathematics, National University of Singapore, 10 Lower Kent Ridge Road, Singapore 119076}
\email{matgwt@nus.edu.sg}
\author{Shuichiro Takeda}
\address{Mathematics Department, University of Missouri, Columbia, 202
  Math Sciences Building, Columbia, MO, 65211}
\email{takedas@missouri.edu}


\date{\today}

\dedicatory{to Professor Roger Howe\\ who started it all \\ on the
  occasion of his 70th birthday}

\begin{abstract}
We give a proof of the Howe duality conjecture in the theory of local theta
correspondence for symplectic-orthogonal or unitary dual pairs
in arbitrary residual characteristic.   
\end{abstract}

\maketitle

\section{\textbf{Introduction}}

Let $F$ be a nonarchimedean local field of characteristic not $2$ and
residue characteristic $p$.  Let $E$ be $F$ itself or a quadratic
field extension of $F$.  For $\epsilon = \pm$, we consider a
$-\epsilon$-Hermitian space $W$ over $E$ of dimension $n$ and an
$\epsilon$-Hermitian space $V$ of dimension $m$.  

\vskip 5pt

Let $G(W)$ and $H(V)$ denote the isometry group of $W$ and $V$
respectively.  Then the group $G(W) \times H(V)$ forms a dual
reductive pair and possesses a Weil representation $\omega_{\psi}$
which depends on a nontrivial additive character $\psi$ of $F$ (and
some other auxiliary data which we shall suppress for now).  To be
precise, when $E=F$ and one of the spaces, say $V$, is odd
dimensional, one needs to consider the metaplectic double cover of
$G(W)$; we shall simply denote this double cover by $G(W)$ as
well. The various cases are tabulated in \cite[\S 3]{gi}.  

\vskip 5pt

In the theory of local theta correspondence, one is interested in the
decomposition of $\omega_{\psi}$ into irreducible representations of
$G(W) \times H(V)$.  More precisely, for any irreducible admissible
representation $\pi$ of $G(W)$, one may consider the maximal
$\pi$-isotypic quotient of $\omega_{\psi}$. This has the form $\pi
\otimes \Theta_{W,V,\psi}(\pi)$ for some smooth representation
$\Theta_{W,V, \psi}(\pi)$ of $H(V)$; we shall frequently suppress $(W,
V,\psi)$ from the notation if there is no cause for confusion.  It was
shown by Kudla \cite{k83} that $\Theta(\pi)$ has finite length
(possibly zero), so we may consider its maximal semisimple quotient
$\theta(\pi)$.  One has the following fundamental conjecture due to
Howe \cite{H1, howe}: 

\vskip 15pt

\noindent\underline{{\bf Howe Duality Conjecture for $G(W) \times
H(V)$}} \vskip 5pt
\noindent (i) $\theta(\pi)$ is either $0$ or irreducible.  \vskip 5pt

\noindent (ii) If $\theta(\pi) = \theta(\pi') \ne 0$, then $\pi =
\pi'$.  \vskip 10pt
 
 \noindent A concise reformulation is: for any irreducible $\pi$ and
 $\pi'$,
\begin{equation}\label{E:HD}
\tag{HD}
\dim \Hom_{H(V)}(\theta(\pi), \theta(\pi'))  \leq \delta_{\pi, \pi'} : =
  \begin{cases} 
  1, \text{  if $\pi \cong \pi'$;} \\
 0, \text{  if $\pi \ncong \pi'$.} \end{cases}
 \end{equation}
 
  \vskip 5pt
  
  We take note of the following theorem: \vskip 5pt

\begin{thm}\label{T:kudla-walds}
\noindent (i) If $\pi$ is supercuspidal, then $\Theta(\pi)$ is either
zero or irreducible (and thus is equal to $\theta(\pi)$).  Moreover,
for any irreducible supercuspidal $\pi$ and $\pi'$,
\[ \Theta(\pi) \cong \Theta(\pi')\ne 0 \Longrightarrow \pi \cong \pi'.
\]

\vskip 5pt

\noindent (ii) $\theta(\pi)$ is multiplicity-free.

\vskip 5pt
\noindent (iii) If $p \ne 2$, the Howe duality conjecture holds.
\end{thm} 

The statement (i) is a classic theorem of Kudla \cite{k83} (see also \cite{mvw}),
whereas (iii) is a well-known result of Waldspurger \cite{w90}. The
statement (ii), on the other hand, is a result of Li-Sun-Tian
\cite{lst}.  We note that the techniques for proving the three
statements in the theorem are quite disjoint from each other. For
example, the proof of (i) is based on arguments using the doubling
see-saw and Jacquet modules of the Weil representation: these have
become standard tools in the study of the local theta correspondence. The
proof of (iii) is based on $K$-type analysis and uses various lattice
models of the Weil representation. Finally, the proof of (ii) is based
on an argument using the Gelfand-Kazhdan criterion for the
(non-)existence of equivariant distributions.  

\vskip 5pt

In this paper, we shall  not assume any of the statements in Theorem
\ref{T:kudla-walds}. Indeed, the purpose of this paper is to give a
simple proof of  the Howe duality conjecture, following a strategy initiated by Howe in his Corvallis article \cite[\S 11]{H1} and using essentially the
same tools in the proof of Theorem \ref{T:kudla-walds}(i) \cite{k83}, as developed further in \cite{mvw}.   Thus, our main theorem is:
\vskip 5pt

\begin{thm}  \label{T:main}
The Howe duality conjecture \eqref{E:HD} holds for the pair $G(W) \times H(V)$.
\end{thm}

\vskip 5pt
Let us make  a few remarks:
\vskip 5pt

\begin{enumerate}[(1)]
\item From our brief sketch of the history of the Howe duality
  conjecture above, the reader can discern two distinct lines of
  attack on the Howe duality conjecture, the genesis of which can both
  be found in \cite{H1}. The first is a study of K-types, which was
  used to establish the conjecture in the archimedean cases and
  adapted to the p-adic case (with $p \ne 2$) by Howe and Waldspurger.
  The second is a doubling see-saw argument  outlined  in \cite[\S
  11]{H1}. Quoting Howe \cite[Pg. 284]{H1}:

\vskip 5pt

``{\em From this doubling construction, we see that the space $(Y \otimes Y^{\vee})_1$ defined above
is described as a $2 \tilde{G}'$-module by Theorem 9.2. By restriction we can investigate its
structure as a $\tilde{G}' \times \tilde{G}'$-module. Doing so we find that the duality conjecture is certainly
true if $G$ or $G'$ is compact, and in general is ``almost'' true. It remains in the
general case to remove the ``almost".} ''
\vskip 5pt

\noindent It has taken almost 40 years to remove the ``almost".

\vskip 10pt

 \item  The above setup makes sense even when $E = F \times F$ is a split
  quadratic algebra, in which case  the groups $G(W)$ and $H(V)$ are
  general linear groups.
   In that case, the Howe duality conjecture has been shown by Minguez in \cite{mi}.
As we shall see, the proof of Theorem \ref{T:main} is essentially
analogous to the one given by Minguez. 

\vskip 5pt
 
 \item In an earlier paper \cite{gt}, we had extended Theorem
   \ref{T:kudla-walds}(i) (with $\Theta(\pi)$ replaced by
   $\theta(\pi)$) from supercuspidal to tempered
   representations. Using this, we had shown the Howe duality
   conjecture for almost equal rank dual pairs. The argument in
   Section \ref{S:non_boundary_case} of this paper
   (doubling see-saw) is the same as that in \cite[\S 2]{gt}, but
   pushed to the limit beyond tempered representations. On the other
   hand, the argument in Section \ref{S:boundary_case} (Kudla's filtration) is entirely
   different from that in \cite[\S 3]{gt} and uses a key technique  of
   Minguez \cite{mi}.  
 
  \item We remark that in the papers \cite{M1,M2,M3, M4},  Mui\'c has
    conducted detailed studies of the local theta correspondence for
    symplectic-orthogonal dual pairs. In \cite{M1}, for example, he
    explicitly determined the theta lift of discrete series
    representations $\pi$ in terms of the Moeglin-Tadi\'c
    classification and observed as a consequence the irreducibility of
    $\theta(\pi)$. The Moeglin-Tadi\'c classification was conditional
    at that point, and we are not sure where it stands today.  In
    \cite{M3, M4}, Mui\'c proved various general properties of the
    theta lifting of tempered representations (such as the issue of
    whether $\Theta(\pi) = \theta(\pi)$), and  obtained very explicit
    information about the theta lifting under the assumption of the
    Howe duality conjecture.  The main tools he used are Jacquet
    modules analysis and Kudla's filtration.   Since the Howe duality
    conjecture is a simple statement without reference to
    classification, it seems desirable to have a classification-free
    proof. Indeed, our result renders most results in \cite{M3,M4}
    unconditional.
  
\end{enumerate}
 \vskip 10pt

  As is well-known, there is another family
of dual pairs associated to quaternionic Hermitian and skew-Hermitian
spaces. (See \cite{w90} or \cite{k94} for more details.) Our proof, unfortunately, does not apply to these
 quaternionic dual pairs, because we have made use of the
 MVW-involution $\pi \mapsto \pi^{MVW}$ on the category of smooth
 representations of $G(W)$ and $H(V)$. For the same reason, the result
 of \cite{lst} in Theorem \ref{T:kudla-walds}(ii)  is not known for
 these quaternionic dual pairs. Unlike the contragredient functor,
 which is contravariant in nature, the MVW-involution is covariant and
 has  the property that $\pi^{MVW}  = \pi^{\vee}$ if $\pi$ is
 irreducible.   It was shown in \cite{sun} that such an involution
 does not exist for quaternionic unitary groups.  
  \vskip 5pt
  
  Nonetheless, even in the quaternionic case, our proof gives a partial result
  which is often sufficient for global applications.  Namely,  if
  $\pi$ is an irreducible Hermitian representation (i.e.\
  $\overline{\pi} = \pi^{\vee}$) and we let
  $\theta_{her}(\pi) \subset \theta(\pi)$ denote the submodule
  generated by irreducible Hermitian summands, then the results of
  Theorem \ref{T:main} hold for Hermitian $\pi$'s and with
  $\theta(\pi)$ replaced by $\theta_{her}(\pi)$.   Namely we have:
  \vskip 5pt
  
\begin{thm}\label{T:main_quaternion}
 Consider a quaternionic dual pair $G(W) \times H(V)$ and let $\pi$ and $\pi'$ be irreducible Hermitian
 representations of $G(W)$.
 Let $\theta_{her}(\pi) \subset \theta(\pi)$ be the submodule generated by irreducible Hermitian summands. Then  we have
\[
\dim \Hom_{H(V)}(\theta_{her}(\pi), \theta_{her}(\pi'))  \leq
\delta_{\pi, \pi'}.
 \]
In particular, if $\pi$ and $\pi'$ are unitary, we have
\[
\dim \Hom_{H(V)}(\theta_{unit}(\pi), \theta_{unit}(\pi'))  \leq
\delta_{\pi, \pi'},
 \]
where $\theta_{unit}(\pi) \subset \theta_{her}(\pi)$ consists of
irreducible unitary summands of $\theta(\pi)$. 
\end{thm}

We give a proof of this theorem in the last section of the paper.

\vskip 15pt
\begin{center}{\bf Acknowledgements}\end{center}
This project was begun  during the authors' participation in the Oberwolfach
workshop ``Modular Forms" in April  2014 and completed while both
authors were participating in the workshop ``The Gan-Gross-Prasad
conjecture" at Jussieu in June-July 2014.  
We thank the Ecole Normale Superieur and the Institut des Hautes
\'Etudes Scientifiques for hosting our respective visits. 
We also thank Goran Mui\'c and Marcela Hanzer for several useful
conversations and email exchanges about \cite{M1, M2,M3,M4} and the
geometric lemma respectively.
We are extremely grateful to Alberto Minguez for explaining to us the
key idea in his paper \cite{mi} for taking care of the representations
on the boundary.
Finally, we thank an anonymous referee who pointed out several
embarrassing errors in a first version of this paper. 
\vskip 5pt

The first author is partially supported by an MOE Tier 1 Grant
R-146-000-155-112 and an MOE Tier Two grant R-146-000-175-112, whereas the
second author is partially supported by NSF grant DMS-1215419.

\vskip 15pt

\section{\bf Basic Notations and Conventions}

\subsection{\bf Fields.}
Throughout the paper, $F$ denotes a nonarchimedean local field of
characteristic different from 2 and residue characteristic $p$. Once
and for all, we fix
a non-trivial additive character $\psi$ on $F$.  Let $E$ be $F$ itself or a quadratic field extension of $F$.
For $\epsilon=\pm 1$, we set
\[ 
\epsilon_0 = \begin{cases} \epsilon \text{ if $E=F$;} \\ 0 \text{
if $E \ne F$.}  \end{cases} 
\] 

\vskip 5pt

\subsection{\bf Spaces.} Let
\begin{align*}
W&=W_n=\text{a $-\epsilon$-Hermitian space over $E$ of dimension $n$ over $E$}\\
V&=V_m=\text{an $\epsilon$-Hermitian space over $E$ of dimension $m$ over $E$}.
\end{align*}
We also set:
\[  s_{m,n}  = \frac{ m - (n+\epsilon_0)}{2}.\]

 \vskip 5pt
 
 \subsection{\bf Groups.}
We will consider the isometry groups associated to the pair $(V,W)$ of
$\pm \epsilon$-Hermitian spaces. More precisely, we set:
\[
G(W)=\begin{cases}\text{the metaplectic group ${\Mp}(W)$, if
    $W$ is symplectic and  $\dim V$ is odd;}\\
\text{the isometry group of $W$, otherwise.}
\end{cases}
\]
We define $H(V)$ similarly by switching the roles of $W$ and
$V$. Occasionally we write
\begin{align*}
G_n&:=G(W_n)\\
H_m&:=H(V_m).
\end{align*}

For the general linear group, we shall write $\GL_n$ for the
group $\GL_n(E)$. Also for a vector space
$X$ over $E$, we write $\det_{\GL(X)}$ or sometimes simply $\det_X$
for the determinant on $\GL(X)$.
 \vskip 5pt
 
 \subsection{\bf Representations.}
 For a $p$-adic group $G$, let ${\rm Rep}(G)$ denote the category of smooth representations of $G$ and  denote
by $\Irr(G)$ the set of equivalence classes of irreducible smooth representations of $G$. 
\vskip 5pt

For a parabolic $P = MN$ of $G$, we have the normalized induction functor
\[  {\rm Ind}_P^G  :  {\rm Rep}(M) \longrightarrow {\rm Rep}(G). \]
On the other hand, we have the normalized Jacquet functor
\[  R_P :  {\rm Rep}(G)  \longrightarrow {\rm Rep}(M).  \]
 If $\overline{P} = M \overline{N}$ denotes the opposite parabolic subgroup to $P$, we likewise have the functor  $R_{\overline{P}}$. 
 We shall frequently exploit the following two Frobenius' reciprocity formulas:
 \[  \Hom_G ( \pi ,  {\rm Ind}_P^G  \sigma)  \cong \Hom_M(R_P(\pi), \sigma)  \quad \text{(standard Frobenius reciprocity)}  \]
  and
  \[  \Hom_G({\rm Ind}_P^G \sigma,  \pi) \cong \Hom_M(\sigma,  R_{\overline{P}}(\pi)) \quad \text{(Bernstein's Frobenius reciprocity)}. \]
Moreover, Bernstein's Frobenius reciprocity is equivalent to the statement:
\[   R_P(\pi^{\vee})^{\vee}  \cong R_{\overline{P}}(\pi) \]
for any smooth representation $\pi$ with contragredient $\pi^{\vee}$, where $\overline{P}$ denotes the opposite parabolic  subgroup to $P$. 
\vskip 5pt

\subsection{\bf Parabolic induction.}
When $G$ is a classical group, we shall
use Tadi\'{c}'s notation for induced representations. Namely, 
 for general linear groups, we set
\[
\rho_1\times\cdots\times\rho_a:=
\Ind_Q^{\GL_{n_1+\dots+n_a}}\rho_1\otimes\cdots\otimes\rho_a
\]
where $Q$ is the standard parabolic subgroup with Levi subgroup
$\GL_{n_1}\times\cdots\times\GL_{n_a}$.  
For a classical group such as $G(W)$, its parabolic subgroups are given as the stabilizers of flags of isotropic spaces.
 If $X_t$ is a $t$-dimensional isotropic space of $W= W_n$ and we decompose
 \[  W  = X_t \oplus W_{n-2t}  \oplus X_t^*, \]
 the corresponding maximal parabolic subgroup
$Q(X_t)= L(X_t) \cdot U(X_t)$  has Levi factor $L(X_t) = \GL(X_t) \times
G(W_{n-2t})$.  If $\rho$ is a representation of $\GL(X_t)$ and
$\sigma$ is a representation of $G(W_{n-2t})$, we write
\[  \rho \rtimes \sigma   =  {\rm Ind}_{Q(X_t)}^{G(W)}  \rho \otimes \sigma. \]
  More generally, a
standard parabolic subgroup $Q$ of $G(W)$  has the Levi factor of the
form $\GL_{n_1}\times\cdots\times\GL_{n_a}\times G(W_{n'})$ and  we set
\[
\rho_1\times\cdots\times\rho_a\rtimes \sigma:=
\Ind_Q^{G(W)}\rho_1\otimes\cdots\otimes\rho_a\otimes \sigma,
\]
where $\rho_i$ is a representation of $\GL_{n_i}$ and $\sigma$ is a
representation of $G(W_{n'})$.   When $G(W) = \Mp(W)$ is a metaplectic
group, we will follow the convention of
\cite[\S 2.2-2.5]{gs} for  normalized parabolic induction.
\vskip 5pt

We have the analogous convention for parabolic subgroups and induced representations of $H(V_m)$.
For example, a maximal parabolic subgroup of $H(V_m)$ has the form  $P(Y_t)=
 M(Y_t) \cdot N(Y_t)$ and is the stabilizer of a $t$-dimensional isotropic subspace $Y_t$ of $V_m$.

\vskip 5pt
To distinguish between representations of $G(W)$ and $H(V)$, 
we will normally use lower
case Greek letters such as $\pi, \sigma$ etc to denote representations of $G(W)$, and upper case Greek letters such as
$\Pi, \Sigma$ etc to denote representations of $H(V)$.
\vskip 5pt

\subsection{\bf Conjugation Action}

Let $X$ be a vector space over $E$ of dimension $n$ and  let $c$ be
the generator of $\operatorname{Gal}(E/F)$. For each
representation $\rho$ of $\GL(X)$, we define the $c$-conjugate
$^c\rho$ of $\rho$ as follows. First, fix a basis of $X$, which
gives an isomorphism $\alpha:\GL(X)\cong\GL_n(E)$. The natural action of $c$
on $\GL_n(E)$ induces an action on $\GL(X)$, which we write as
$c\cdot g:=\alpha^{-1}\circ c\circ\alpha(g)$ for $g\in\GL(X)$. 

Then we define $^c\rho$ by
\[
^c\rho(g):=\rho(c\cdot g)
\]
for $g\in \GL(X)$. Of course, a different choice of the basis gives a
different  $\alpha$ and thus a different automorphism $g \mapsto c \cdot g$ of $\GL(X)$.
But these different automorphisms differ from each other by  inner automorphisms of $\GL(X)$,
and hence the isomorphism class of
$^c\rho$ is independent of the choice of the basis. Of course if $E=F$, then $^c\rho=\rho$. 

\vskip 5pt

\subsection{\bf MVW}
In   \cite[p. 91]{mvw},
Moeglin, Vigneras and Waldspurger introduced a  functor
\[  MVW:  {\rm Rep}(G(W))  \longrightarrow {\rm Rep}(G(W)) \]
which is an involution and satisfies
\[  \pi^{MVW}  = \pi^{\vee}  \quad \text{ if $\pi$ is irreducible.} \]
Unlike the contragredient functor, this  MVW involution  is covariant.
It will be useful to observe that 
\begin{equation}\label{E:MVW}
  (\rho \rtimes \sigma)^{MVW}  = {^c}\rho \rtimes \sigma^{MVW}.
\end{equation}

\vskip 5pt

\subsection{\bf Weil representations.}
To consider the Weil representation of the pair $G(W)\times H(V)$, we
need to specify extra data to give a splitting $G(W) \times
H(V)\rightarrow \Mp(W\otimes V)$ of the dual pair. 
 Such splittings were constructed and parametrized by Kudla  \cite{k94} and we shall use his 
 convention here, as described in \cite[\S 3.2-3.3]{gi}.
 In particular, a splitting is specified by  fixing a
pair of splitting characters $\mathbold{ \chi} = (\chi_V, \chi_W)$, which are
certain unitary characters of $E^{\times}$. Pulling back the Weil representation
of $\Mp(W\otimes V)$ to $G(W)\times H(V)$ via the
splitting, we obtain the  associated Weil
representation $\omega_{W, V, \mathbold{\chi}, \psi}$ of $G(W)\times H(V)$.  
  Note that the character $\chi_V$ satisfies 
\begin{equation}\label{E:chi}
  {^c}\chi_V^{-1}  =  \chi_V
\end{equation}
and likewise for $\chi_W$.  We shall frequently
suppress $\mathbold{\chi}$ and $\psi$ from the notation, and simply
write $\omega_{W, V}$ for the Weil representation.

 \vskip 15 pt

\section{\bf Some Lemmas}
In this section, we record a couple of standard lemmas which will be used in the proof of the main theorem. We also introduce the important notion of ``occurring on the boundary".
\vskip 5pt

\subsection{\bf Kudla's filtration.}  The first lemma describes the computation of the Jacquet modules of the Weil representation.
 This is a well-known result of Kudla \cite{k83}; see  also \cite{mvw}.
\vskip 5pt

\begin{lem} \label{L:kudla}
  The Jacquet module $R_{Q(X_a)}(\omega_{W_n, V_{m}} )$ has
an equivariant filtration
\[ R_{Q(X_a)}(\omega_{W_n,V_{m}}) = R^0 \supset R^1 \supset \cdots \supset R^a
\supset R^{a+1} = 0 \] 
whose successive quotients $J^k = R^k/R^{k+1}$
are described in \cite[Lemma C.2]{gi}. 
More precisely,
\begin{align*}
 J^k &= \Ind^{\GL(X_a) \times G(W_{n-2a}) \times H(V_m)}_{Q(X_{a-k}, X_a)
\times G(W_{n-2a}) \times P(Y_k)}\\
&\qquad\left(\chi_V
|{\det}_{X_{a-k}}|^{\lambda_{a-k}} \otimes
C^{\infty}_c(\Isom_{E,c}(X_k, Y_k)) \otimes
\omega_{W_{n-2a}, V_{m-2k}}\right),
\end{align*}
 where
\begin{itemize}
\item[\textperiodcentered] $\lambda_{a-k} = s_{m,n} + \frac{a-k}{2} $;
\item[\textperiodcentered] $V_m = Y_k + V_{m-2k} +Y_k^*$ with $Y_k$ a $k$-dimensional
isotropic space;
\item[\textperiodcentered] $X_a = X_{a-k} + X_k'$ and $Q(X_{a-k}, X_a)$ is the maximal
parabolic subgroup of $\GL(X_a)$ stabilizing $X_{a-k}$;
\item[\textperiodcentered] $\Isom_{E,c}(X_k, Y_k)$ is the set of $E$-conjugate-linear isomorphisms from $X_k$ to $Y_k$;
\item[\textperiodcentered] $\GL(X_k)\times\GL(Y_k)$ acts on
  $C_c^\infty(\Isom_{E,c}(X_k,Y_k))$ as 
  \[ ((b,c)\cdot f)(g)=\chi_V(\det b)\chi_W(\det
  c)f(c^{-1}g b) \]
   for $(b,c)\in \GL(X_k)\times\GL(Y_k)$, $f\in
  C_c^\infty(\Isom_{E,c}(X_k,Y_k))$ and $g\in\GL_k$.
\item [\textperiodcentered] $J^k=0$ for $k>\min\{a, q\}$, where $q$ is
  the Witt index of $V_m$.
\end{itemize} 
In particular, the bottom piece
of the filtration (if nonzero) is:
\[ J^a \cong \Ind^{\GL(X_a) \times G(W_{n-2a}) \times H(V_m)}_{\GL(X_a)
\times G(W_{n-2a}) \times P(Y_a)} \left(C^{\infty}_c(\Isom_{E,c}(X_a,Y_a)) \otimes \omega_{W_{n-2a},
V_{m-2a}}\right). \]
\end{lem}

\vskip 5pt

\subsection{\bf A degenerate principal series.}
Another key ingredient used later is the decomposition of a certain
degenerate principal series representation. Consider the ``doubled"
space $W +  W^-$. This ``doubled" space contains the diagonally embedded 
\[  \Delta W \subset W + W^- \]
as a maximal isotropic subspace whose stabilizer
$Q(\Delta W)$ is a maximal parabolic subgroup of $G(W+ W^-)$ and has Levi factor $\GL(\Delta W)$. We may consider the degenerate principal series representation
\[  I(s)  =   {\rm Ind}_{Q(\Delta W)}^{G(W+W^-)} \chi_V |\det|^{s} \]
of $G(W + W^-)$  induced from the
character $\chi_V|\det|^s$ of $Q(\Delta W)$ (normalized induction).

  \vskip 5pt

 The following  lemma (see \cite{kr05}) describes the restriction of $I(s)$ to the subgroup $G(W) \times G(W)  \subset G(W + W^-)$.
\vskip 5pt

\begin{lem} \label{L:key} 
As a representation of $G(W) \times
G(W^-)$, $I(s)$
possesses an equivariant filtration
\[ 0 \subset I_0(s) \subset I_1(s) \subset\cdots\subset I_{q_W}(s) = {\rm
Ind}_{Q(\Delta W)}^{G(W+W^-)} \chi_V \cdot |\det|^{s} \] whose successive
quotients are given by
\begin{align}
 R_t(s) &= I_t(s) / I_{t-1}(s)   \notag \\
 &=     {\rm Ind}_{Q_{t} \times Q_t}^{G(W) \times
G(W^-)} \Bigg( \left(\chi_V |{\det}_{X_{t}}|^{s + \frac{t}{2}} \boxtimes
\chi_V |{\det}_{X_t}|^{s+\frac{t}{2}} \right) \\
&\hspace{1.5in}\otimes \left(  (\chi_V \circ {\det}_{W_{n-2t}^-}) \otimes C^{\infty}_c(
G(W_{n-2t})) \right) \Bigg). \notag 
\end{align} 
Here, the induction is normalized and
\begin{itemize}
\item[\textperiodcentered] $q_W$ is the Witt index of $W$;
\item[\textperiodcentered] $Q_t$ is the maximal parabolic subgroup of $G(W)$ stabilizing a
$t$-dimensional isotropic subspace $X_t$ of $W$, with Levi subgroup
$\GL(X_t) \times G(W_{n-2t})$, where $\dim W_{n-2t} = n-2t$.
\item[\textperiodcentered] $G(W_{n-2t})\times G(W_{n-2t})$ acts on $C^{\infty}_c(
G(W_{n-2t}))$ by  left-right translation.
\end{itemize} 
In particular,
\[ R_0 = R_0(s) =  (\chi_V \circ {\det}_{W^-}) \otimes C^{\infty}_c(G(W)) \]
is independent of $s$.
\end{lem} 

\vskip 5pt

\subsection{\bf Boundary.}
The above lemma suggests a  notion which plays an important role in this paper.
For $\pi \in {\rm Irr}(G(W))$, we shall say that $\pi$ occurs on the
boundary of $I(s)$ if there exists $0<t \leq q_W$ such that
\[  \Hom_{G(W)}(R_t(s),  \pi)  \ne 0. \]
 By Bernstein's Frobenius reciprocity, this is equivalent to
 \begin{equation}  \label{E:boundary2}
  \Hom_{\GL(X_t)}(\chi_V  |\det|^{s  + \frac{t}{2}},
  R_{\overline{Q}(X_t)}(\pi) )\ne 0, \end{equation}
where $\overline{Q}(X_t)  = L(X_t)  \cdot \overline{U}(X_t)$ stands
for the parabolic subgroup opposite to $Q(X_t)$.
Dualizing and using the MVW involution, this is in turn equivalent to   
 \begin{equation} \label{E:boundary}
 \pi  \hookrightarrow   \chi_V  |{\det}_{\GL(X_t)}|^{-s   -
   \frac{t}{2}}   \rtimes \sigma  \end{equation}
 for some irreducible representation $\sigma$ of $G(W_{n-2t})$. 
 This terminology  is due to
 Kudla and Rallis and the reader may consult  \cite[Definition
 1.3]{kr05} for the explanation of the use of ``boundary''.

\vskip 10pt

\subsection{\bf Outline of proof.}\label{S:outline}
With the basic notations introduced, we can now give a brief outline
of the proof of Theorem \ref{T:main}.
First of all, there is no loss of generality in
assuming that 
\begin{equation}\label{E:assumption}
m \leq n+\epsilon_0, 
\end{equation}
so that
\begin{equation}\label{E:s_m_n}
 s_{m,n}  = \frac{ m - (n+\epsilon_0)}{2} \leq 0,
\end{equation}
because otherwise one can switch
the roles of $G(W)$ and $H(V)$. We shall assume this henceforth.
\vskip 5pt

The proof proceeds by induction on $\dim W$, with the base case with $\dim W = 0$ being trivial.
The inductive step is divided into two different parts.  
\vskip 5pt

The first part, which is
given in Section \ref{S:non_boundary_case}, deals with the case when
$\pi$ does not occur on the boundary of $I(-s_{m,n})$, which cover ``almost all''
representations.   The argument for this case has been sketched by
Howe  \cite[\S 11]{H1}, using a see-saw dual pair and Lemma
\ref{L:key} above. 
 For this part, the induction hypothesis is not used.
\vskip 10pt

The second part of the proof, which is given in Section
\ref{S:boundary_case}, deals with the case when $\pi$ occurs on the
boundary of $I(-s_{m,n})$.  To use the induction hypothesis, we use
Lemma \ref{L:kudla} and   a key idea of Minguez  \cite{mi}
in his proof of the Howe duality conjecture for general linear groups. 
\vskip 5pt

In Section \ref{S:assem}, we assemble the results of Section \ref{S:non_boundary_case} and Section \ref{S:boundary_case} to complete the proof of Theorem \ref{T:main}.

\vskip 15 pt

\section{\bf A See-Saw Argument} \label{S:non_boundary_case}
In this section, we will give a proof of \eqref{E:HD} when at least one of  $\pi$ or $\pi' \in\Irr(G(W))$ does
not occur on the boundary of $I(-s_{m,n})$ by assuming
\eqref{E:assumption} or equivalently \eqref{E:s_m_n}.
As is mentioned at the end of the last section, though we prove
\eqref{E:HD} by induction on $\dim W$, it turns out that for this
case, one can prove \eqref{E:HD} without appealing to the induction
hypothesis. Namely, we shall prove
  
\begin{thm} \label{T:non_boundary} 
Assume that $m \leq n+\epsilon_0$ and suppose that $\pi \in {\rm
  Irr}(G(W))$ does not occur on the boundary of $I(-s_{m,n})$. Then for
any $\pi' \in {\rm Irr}(G(W))$,
   \[  \dim \Hom_{H(V)}(\theta(\pi), \theta(\pi'))  \leq \delta_{\pi, \pi'}.  \]
  In particular, $\theta(\pi)$ is either zero or irreducible, and
  moreover for  any irreducible $\pi'$,
\[ 
0 \ne \theta(\pi) \subset \theta(\pi') \Longrightarrow \pi \cong\pi'.
\]
 \end{thm}
\begin{proof}
First, we consider the following see-saw diagram
\[ 
\xymatrix{ G(W+W^-)\ar@{-}[d]_{}\ar@{-}[dr]&H(V)\times
H(V)\ar@{-}[d]\\ G(W)\times G(W^-)\ar@{-}[ur]&H(V)^{\Delta}, }
\] 
where $W^{-}$ denotes the space obtained from $W$ by multiplying
the form by $-1$, so that $G(W^-) = G(W)$.  
Given  irreducible representations $\pi$ and $\pi'$ of $G(W)$,  the see-saw identity
\cite[\S 6.1]{gi} gives:
\begin{align}\label{E:see-saw}
&\Hom_{G(W) \times G(W)} ( \Theta_{V, W + W^-}(\chi_W), \pi' \otimes
\pi^{\vee}\chi_V)\\
=&\Hom_{H(V)^\Delta}( \Theta(\pi') \otimes
\Theta(\pi)^{MVW}, \CC). \notag
\end{align}
 Here $\Theta_{V,W + W^-}(\chi_W)$ denotes the big theta lift of the
character $\chi_W$  of $H(V)$ to $G(W+ W^-)$. 
\vskip 5pt

We analyze each side of the see-saw identity in turn. 
For the RHS,  one has
\begin{align}\label{E:see-saw2}
  \Hom_{H(V)^\Delta}( \Theta(\pi') \otimes
\Theta(\pi)^{MVW}, \CC) 
\supseteq &\Hom_{H(V)^\Delta}( \theta(\pi') \otimes
\theta(\pi)^{\vee}, \CC) \\
= &\Hom_{H(V)}( \theta(\pi'),  \theta(\pi)). \notag
\end{align}

\vskip 5pt

For the LHS of the see-saw identity \eqref{E:see-saw}, we need to
understand $\Theta_{V,W + W^-}(\chi_W)$. 
 It is known that  $\Theta_{V,W + W^-}(\chi_W)$ is irreducible  (see
 \cite[Prop. 7.2]{gi}). Moreover, it was shown by Rallis that
\begin{equation}\label{E:Rallis}
\Theta_{V, W + W^-}(\chi_W) \hookrightarrow I(s_{m,n}) =  {\rm Ind}_{Q(\Delta
W)}^{G(W+W^-)} \chi_V |\det|^{s_{m,n}}.
\end{equation}
  Since  $s_{m,n} \leq 0$,  there is a surjective map (see \cite[Prop. 8.2]{gi})
\[ I(-s_{m,n})  \longrightarrow \Theta_{V, W + W^-}(\chi_W). \] 
Hence the see-saw identity \eqref{E:see-saw} gives:
\[ \Hom_{G(W) \times G(W)}(  I(-s_{m,n}), \pi' \otimes \pi^{\vee}\chi_V) \supseteq \Hom_{H(V)}
(\theta(\pi') , \theta(\pi)).  \] 
 To prove the theorem, it suffices to prove that the LHS has dimension $\leq 1$
 with equality only if $\pi \cong \pi'$.
 \vskip 5pt

Applying Lemma \ref{L:key},  we claim that if $\pi$ is not
on the boundary of $I(-s_{m,n})$, the natural restriction map
\[ \Hom_{G(W) \times G(W)}(  I(-s_{m,n}), \pi' \otimes \pi^{\vee}\chi_V ) \longrightarrow \Hom_{G(W)
\times G(W)}( R_0, \pi' \otimes \pi^{\vee}\chi_V) \]
 is injective. This will
imply the theorem since the RHS has dimension $\leq 1$, with equality
if and only if $\pi = \pi'$.  \vskip 5pt

To deduce the claim, it suffices to to show that for each $0 < t \leq q_W$,
\begin{equation}\label{E:hom_space}
 \Hom_{G(W) \times G(W)}( R_t(-s_{m,n}) , \pi' \otimes \pi^{\vee}\chi_V) = 0.
\end{equation}
By Bernstein's Frobenius reciprocity, this Hom space is equal to
\begin{align*} 
&\Hom_{L(X_t) \times L(X_t)} \Big( \big( \chi_V |\det|^{-s_{m,n} +
\frac{t}{2}} \boxtimes\chi_V |\det|^{-s_{m,n}+\frac{t}{2}} \big)\\
&\hspace{70pt} \otimes \left( (\chi_V \circ {\det}_{W_{n-2t}^-})
  \otimes C^{\infty}_c( G(W_{n-2t}) )\right),\; R_{\overline{Q}_t}(\pi')
\otimes R_{\overline{Q}_t} (\pi^{\vee}\chi_V)\Big).
\end{align*}
 Hence we deduce that the equation \eqref{E:hom_space} holds if
\[  \Hom_{\GL(X_t)}(\chi_V  |\det|^{- s_{m,n}  + \frac{t}{2}},  R_{\overline{Q}_t}(\pi^{\vee}\chi_V) )=  0. \]
By dualizing and Bernstein's Frobenius reciprocity, this is equivalent to
\[  \Hom_{\GL(X_t)}(R_{Q_t}(\pi \cdot (\chi_V^{-1}\circ {\det}_W)) , (\chi_V^{-1} \circ {\det}_{\GL(X_t)}) \cdot  |\det|^{s_{m,n} - \frac{t}{2}}) = 0. \]
Now note that
\[  \chi_V \circ {\det}_{W} |_{\GL(X_t)} = \chi_V^2 \circ {\det}_{\GL(X_t)}. \]
Hence the above condition is equivalent to 
\[  \Hom_{\GL(X_t)}(R_{Q_t}(\pi) , \chi_V |\det|^{s_{m,n} - \frac{t}{2}}) = 0. \]
Since this condition holds when $\pi$ is not on the
boundary of $\omega_{W,V}$, the equation \eqref{E:hom_space} is
proved. This completes the proof of Theorem \ref{T:non_boundary}.
\end{proof}
\vskip 10pt

Finally, let us note that the above argument
also gives the following proposition, which we will use later.
\vskip 5pt

\begin{prop} \label{P:non_boundary}
Assume $m\leq n+\epsilon_0$. If $\pi\ne \pi'\in{\rm Irr}(G(W))$ are such that 
\begin{equation}\label{E:prop}
 \Hom_{H(V)}(\theta_{W, V}(\pi) , \theta_{W, V}(\pi')) \ne 0, 
\end{equation}
 then there exists $t>0$ such that   
\[  \pi \hookrightarrow \chi_V |{\det}_{\GL(X_t)}|^{s_{m,n} - \frac{t}{2}} \rtimes \tau\]
and
\[  \pi'\hookrightarrow \chi_V |{\det}_{\GL(X_t)}|^{s_{m,n} - \frac{t}{2}} \rtimes \tau'\]
for some $\tau$ and $\tau'\in {\Irr}(G(W_{n-2t}))$.
\end{prop}
\begin{proof}
If $\pi\ne \pi'$ but $\Hom_{H(V)}(\theta_{W, V}(\pi) ,
\theta_{W, V}(\pi')) \ne 0$, then arguing as above one
must have $\Hom_{G(W) \times G(W)}( R_t(-s_{m,n}) , \pi\otimes
\pi'^{\vee}\chi_V)\neq 0$ for some $t>0$, which gives the conclusion
of the proposition. 
\end{proof}
 
Of course, the hypothesis of this proposition contradicts the Howe
duality \eqref{E:HD}, and hence is never satisfied. So what this
proposition is saying is that if \eqref{E:HD} is to be violated as in
\eqref{E:prop}, it must be violated by representations $\pi$ and $\pi'$ which occur on the
boundary for ``the same $t$''.  

 \vskip 15pt

\section{\bf Life on the Boundary}\label{S:boundary_case}
After Theorem \ref{T:non_boundary}, we see that 
to prove Theorem \ref{T:main} or equivalently \eqref{E:HD}, it remains
to consider the case when both  $\pi$ and $\pi'$  occur
on the boundary of $I(-s_{m,n})$, as in Proposition
\ref{P:non_boundary}. In this section, we examine what life looks like
on the boundary. We continue to assume that $m \leq n + \epsilon_0$, or equivalently that $s_{m,n} \leq 0$ (see
\eqref{E:assumption}  and \eqref{E:s_m_n}).

\vskip 5pt

\subsection{\bf An idea of Minguez.}
Since $\pi$ occurs on the boundary of $I(-s_{m,n})$, we have
\[  \pi \hookrightarrow  \chi_V  |{\det}_{\GL_{(X_t)}}|^{s_{m,n} - \frac{t}{2}} \rtimes  \pi_0  \]
for some $\pi_0 \in {\rm Irr}(G(W_{n-2t}))$ and some $t>0$.
By induction in stages, we have
\[ \pi \hookrightarrow  (\chi_V |-|^{s_{m,n} - t + \frac{1}{2}} \times
\cdots\times \chi_V | -|^{s_{m,n} -\frac{1}{2}}) \rtimes \pi_0. \]
For simplicity, let us set
\[  s_{m,n,t}  = s_{m,n} - t + \frac{1}{2} < 0. \]
 \vskip 5pt
 
 Now let us use a key idea of Minguez \cite{mi}.  Let $a>0$ be maximal such that
\[  \pi \hookrightarrow   \chi_V | -|^{s_{m,n,t}} \cdot 1^{\times a}   \rtimes \sigma,\]
where
\[
1^{\times a}=\underbrace{1\times\cdots\times 1}_{a\;
  \text{times}}=\Ind_B^{\GL_a}1\otimes\cdots\otimes 1
\]
and  $\sigma$ is an irreducible representation of $G(W_{n-2a})$.
 To simplify notation, let us set
\begin{equation}\label{E:sigma_a}
 \sigma_a :=  \chi_V | -|^{s_{m,n,t}} \cdot 1^{\times a}   \rtimes \sigma.
\end{equation}
Note that the representation $1^{\times a}$ is irreducible and generic by \cite[Theorem 9.7]{z80}.
 \vskip 5pt
 
 Observe that if $\pi \hookrightarrow \sigma_a$ with $a>0$ maximal, then
\[  \sigma\nsubseteq   \chi_V |-|^{s_{m,n,t}} \rtimes \sigma' \]
for any $\sigma'$.  In fact, the converse is also true, as we shall
verify in Corollary \ref{C:max}(ii) below.

\vskip 5pt

The chief reason for considering $\sigma_a$ with $a>0$ maximal is the following proposition:

\begin{prop}\label{P:unique1}
 Suppose $\sigma_a = \chi_V | -|^{s_{m,n,t}} \cdot
  1^{\times a}   \rtimes \sigma$ is such that
   \[ \sigma\nsubseteq   \chi_V
  |-|^{s_{m,n,t}} \rtimes \sigma' \]
     for any $\sigma'$. Then $\sigma_a$  has a unique
  irreducible submodule.
\end{prop}

The proof of Proposition \ref{P:unique1} relies on the
following key technical lemma, which we state in slightly greater
generality as it may be useful for other purposes. The proof of the lemma is deferred to the appendix.
\vskip 5pt

\begin{lem}  \label{L:geometric}
Let $\rho$ be a supercuspidal representation of $\GL_r(E)$ and
consider the induced representation 
\[
\sigma_{\rho,a} = \rho^{\times a}  \rtimes \sigma
\] 
of $G(W_n)$ where $\sigma$ is an irreducible  representation  of
$G(W_{n-ra})$. Assume that
\begin{itemize}
\item[(a)]  $^c\rho^{\vee}  \ne \rho$;
\item[(b)]  $ \sigma \nsubseteq \rho \rtimes \sigma_0$ for any
  $\sigma_0$.
\end{itemize}
 Then we have the following:
\begin{enumerate}[(i)]
\item One has a natural short exact sequence
\[  
\begin{CD}
0 @>>> T @>>>  R_{\overline{Q}(X_{ra})} \sigma_{\rho, a} @>>>
({^c\rho}^{\vee})^{\times a}  \otimes \sigma @>>> 0 \end{CD} 
\]
and $T$ does not  contain any irreducible subquotient of the form
$({^c\rho}^{\vee})^{\times a} \rtimes \sigma'$ for any $\sigma'$. 
In particular, $R_{\overline{Q}(X_{ra})} \sigma_{\rho, a}$
contains $({^c\rho}^{\vee})^{\times a}  \otimes \sigma$ with multiplicity one, and
does not contain any other subquotient of the form $({^c\rho}^{\vee})^{\times a}  \otimes \sigma'$.
Likewise, $R_{Q(X_{ra})} \sigma_{\rho, a}$
contains $\rho^{\times a}  \otimes \sigma$ with multiplicity one and does not contain any other subquotient of the form
$\rho^{\times a}  \otimes \sigma'$. 

\vskip 5pt

\item  The induced representation $\sigma_{\rho, a}$
has a unique irreducible submodule.
\end{enumerate}
\end{lem}
\vskip 5pt

\begin{proof}[Proof of Proposition \ref{P:unique1}]
We shall apply this lemma with
\[ r=1, \quad \rho  = \chi_V |-|^{s_{m,n,t}}  \quad \text{and}  \quad
\sigma_a = \chi_V | -|^{s_{m,n,t}} \cdot 1^{\times a}   \rtimes
\sigma. \]
Here, condition (a) holds since $s_{m,n,t} <0$, whereas condition (b)
holds by the maximality of $a$. This proves Proposition
\ref{P:unique1}.
\end{proof}

\noindent We shall have another occasion to use Lemma \ref{L:geometric} later on. We also note the following corollary:
\vskip 5pt

\begin{cor}  \label{C:max}
Suppose that $\pi \hookrightarrow \sigma_a$ and $\sigma \nsubseteq \chi_V |-|^{s_{m,n,t}} \rtimes \sigma'$ for any $\sigma'$. 
\vskip 5pt
\begin{enumerate}[(i)]  
\item If $\pi \hookrightarrow  \delta_a :=  \chi_V | -|^{s_{m,n,t}} \cdot 1^{\times a}   \rtimes
\delta$ for some $\delta$, then $\delta \cong \sigma$.
\vskip 5pt
\item Moreover, $a$ is maximal with respect to the property that $\pi
  \hookrightarrow \delta_a$ for some irreducible $\delta$. 
\end{enumerate}
 
\end{cor}
\vskip 5pt

\begin{proof}
 By the exactness of the Jacquet functor, $R_{Q(X_a)}(\pi)$ is a submodule of $R_{Q(X_a)}(\sigma_a)$. By Lemma \ref{L:geometric}, it follows that $R_{Q(X_a)}(\pi)$ contains 
$ \chi_V | -|^{s_{m,n,t}} \cdot  1^{\times a}   \otimes \sigma$ with multiplicity one, and does not contain any other subquotient of the form $ \chi_V | -|^{s_{m,n,t}} \cdot
  1^{\times a}   \otimes \sigma'$ with $\sigma' \ne \sigma$. This key fact will imply both (i) and (ii).
  \vskip 5pt
  
 \noindent (i)  If $\pi \hookrightarrow \delta_a$, then $R_{Q(X_a)}(\pi)$ contains $ \chi_V | -|^{s_{m,n,t}} \cdot
  1^{\times a}   \otimes \delta$ as a quotient. By the key fact observed above, it follows that $\delta \cong \sigma$.
 
  \vskip 5pt
  
  \noindent (ii)  Suppose for the sake of contradiction that 
  \[  \pi \hookrightarrow  \chi_V | -|^{s_{m,n,t}} \cdot  1^{\times (a+1)}   \rtimes \sigma'  \]
     for some irreducible $\sigma'$. Then by induction in stages, one has
     \[  \pi \hookrightarrow  \chi_V | -|^{s_{m,n,t}} \cdot  1^{\times a}   \rtimes \delta  \quad \text{ with $\delta =   \chi_V | -|^{s_{m,n,t}} \rtimes \sigma'$}. \]
By the Frobenius reciprocity, one has a nonzero equivariant map
\[  R_{Q(X_a)}(\pi) \longrightarrow  \chi_V | -|^{s_{m,n,t}} \cdot  1^{\times a}  \otimes \delta. \]
By the key fact observed above, the image of this nonzero map must be isomorphic to $ \chi_V | -|^{s_{m,n,t}} \cdot  1^{\times a}   \otimes \sigma$.
Hence,  
\[  \sigma \hookrightarrow \delta = \chi_V | -|^{s_{m,n,t}} \rtimes \sigma'  \]
which is a contradiction to the hypothesis of the corollary.
\end{proof}

\vskip 15pt

\subsection{\bf A key computation.}
We are now ready to launch into a computation needed to complete the proof of Theorem \ref{T:main} for the representations on the boundary. 
The following is the key proposition:

\begin{prop}\label{P:unique2}
Assume $0 \ne \Pi \subset \theta(\pi)$ and $\Pi$ is irreducible.
\vskip 5pt

\noindent (i)  If
\[  \pi \hookrightarrow    \chi_V | -|^{s_{m,n,t}} \cdot 1^{\times a}   \rtimes \sigma  \]
with $a$ maximal (and for some $\sigma$), then
\[    \Pi \hookrightarrow   \chi_W |-|^{s_{m,n,t}} \cdot 1^{\times a} \rtimes \Sigma  \]
for some $\Sigma$ and  where $a$ is also maximal for $\Pi$. 

\vskip 5pt
\noindent (ii) Moreover, whenever $\Pi$ is presented as a submodule as above, one has
\[   
0  \ne \Hom_{G_n\times H_m}( \omega_{W_n, V_m},  \pi \otimes \Pi)
\hookrightarrow \Hom_{G_{n-2a}\times H_{m-2a}}
(\omega_{W_{n-2a}, V_{m-2a}} , \sigma \otimes\Sigma),
\]
 so  that $\Sigma\subseteq\theta_{W_{n-2a}, V_{m-2a}}(\sigma)$.
\end{prop}
\vskip 5pt

 \vskip 5pt

\begin{proof}
(i) Since $0\neq \Pi\subseteq\theta(\pi)$, we have
\begin{align*}
 0 &\ne \Hom_{G_n\times H_m}( \omega_{W,V}, \pi  \otimes \Pi)\\
 &\hookrightarrow \Hom_{G_n\times H_m}(\omega_{W,V}, \sigma_a \otimes \Pi)\\
  &= \Hom_{\GL(X_a)\times G_{n-2a}\times H_m}
( R_{Q(X_a)}(\omega_{W,V}), \chi_V|-|^{s_{m,n,t}} 1^{\times a}\otimes \sigma \otimes \Pi),
  \end{align*}
where we used the Frobenius reciprocity for the last step. Now
the Jacquet module $R_{Q(X_a)}(\omega_{W,V}) $ of the Weil
representation is computed as in the lemma, which
 implies that there is a natural restriction map
\begin{align*}
 &\Hom_{\GL(X_a)\times G_{n-2a}\times H_m}
( R_{Q(X_a)}(\omega_{W,V}), \chi_V|-|^{s_{m,n,t}} \cdot 1^{\times
   a} \otimes \sigma \otimes \Pi)\\ 
&\qquad\longrightarrow \Hom_{\GL(X_a)\times G_{n-2a}\times H_m} 
( J^a, \chi_V|-|^{s_{m,n,t}} \cdot 1^{\times a} \otimes \sigma \otimes\Pi).
\end{align*}
 We claim that this map is injective.  To see this, it suffices to show that for all $0 \leq k < a$,
\[  \Hom_{\GL(X_a)\times G_{n-2a}\times H_m}
( J^k,   \chi_V|-|^{s_{m,n,t}} \cdot 1^{\times a} \otimes \sigma \otimes \Pi) = 0. \]
By the above lemma, this Hom space is equal to
\begin{align*}   
&\Hom_{M(X_{a-k}, X_a)\times G_{n-2a}\times H_m}\\
&\qquad\Big(\Ind^{H(V_m)}_{P(Y_k)}\chi_V |{\det}_{X_{a-k}}|^{\lambda_{a-k}}
\otimes C^{\infty}_c(\Isom_{E,c}(X_k,Y_k))\otimes
\omega_{W_{n-2a}, V_{m-2k}},\\
&\hspace{2in}R_{\overline{Q}(X_{a-k}, X_a)}(
\chi_V|-|^{s_{m,n,t}} \cdot  1^{\times a}) \otimes \sigma \otimes
\Pi\Big),
\end{align*}
where $M(X_{a-k}, X_a)$ is the Levi factor of the parabolic subgroup of $\GL(X_a)$
stabilizing $X_{a-k}$. Because $1^{\times a}$ is generic, the second
representation in this Hom
space has a nonzero Whittaker functional when viewed as a
representation of $\GL(X_{a-k})$ and hence the first one must also
have a non-zero Whittaker functional, which is possible only when
$a-k=1$. Therefore, if this Hom space is nonzero, we must have
$a-k=1$. But in that case, one has:
\[  
\lambda_1 =   s_{m,n} + \frac{1}{2}  >  s_{m,n} - t + \frac{1}{2} =
s_{m,n,t}, 
\]
so that the above Hom space is zero even when $a-k=1$.
\vskip 5pt

Therefore we have $J^a\neq 0$ and 
\begin{align*}
 0 \ne &\Hom_{\GL(X_a)\times G_{n-2a}\times H_m}
( J^a,   \chi_V|-|^{s_{m,n,t}} \cdot 1^{\times a} \otimes\sigma \otimes \Pi)\\
=& \Hom_{H_m} 
( \chi_W |-|^{-s_{m,n,t}} \cdot 1^{\times  a} \rtimes \Theta_{W_{n-2a}, V_{m-2a}}(\sigma),  \Pi).
\end{align*}
Dualizing and applying MVW along with \eqref{E:MVW} and \eqref{E:chi}, this shows that 
\[  \Pi  \hookrightarrow   \chi_W |-|^{s_{m,n,t}} \cdot 1^{\times a}
\rtimes (\Theta_{W_{n-2a}, V_{m-2a}}(\sigma)^\vee)^{MVW}, \]
and hence
\begin{equation}  \label{E:Pi}
  \Pi  \hookrightarrow   \chi_W |-|^{s_{m,n,t}} \cdot 1^{\times a} \rtimes \Sigma \end{equation}
for some irreducible representation $\Sigma$ of $H(V_{m-2a})$ which is
a subquotient of the representation $ (\Theta_{W_{n-2a}, V_{m-2a}}(\sigma)^\vee)^{MVW}$
and hence of $\Theta_{W_{n-2a}, V_{m-2a}}(\sigma)$.
\vskip 5pt

To prove (i), it remains to show that in \eqref{E:Pi},  the integer $a$ is maximal for $\Pi$. 
Let $b \geq a$ be maximal such that
\[  \Pi  \hookrightarrow  \chi_W |-|^{s_{m,n,t}} \cdot 1^{\times b} \rtimes \Sigma_0  \] 
for some irreducible representation $\Sigma_0$ of $H(V_{m-2b})$. Then
we have
\begin{align}  \label{E:kudla}
 0 &\ne  \Hom_{G_n\times H_m}( \omega_{W_n, V_m},  \pi \otimes \Pi)   \\
 &\hookrightarrow \Hom_{G_n\times H_m}( \omega_{W_n, V_m},  \pi
 \otimes (\chi_W |-|^{s_{m,n,t}} \cdot 1^{\times b} \rtimes \Sigma_0
 ) ) \notag \\
 & = \Hom_{G_{n}\times\GL(Y_b)\times H_{m-2b}}(R_{P(Y_b)}(\omega_{W_n,
   V_m}),  \pi \otimes  \chi_W
 |-|^{s_{m,n,t}} \cdot 1^{\times b} \otimes \Sigma_0 ). \notag 
\end{align}
We can compute the Jacquet module $R_{P(Y_b)}(\omega_{W_n, V_m})$ by
using Lemma \ref{L:kudla} with the roles of $H(V_m)$ and $G(W_n)$
switched. But for this, it should be noted that the exponent  $\lambda_{b-k}$ (for $k < b$)
 in Lemma \ref{L:kudla} satisfies:
 \begin{equation}\label{E:lambda}
   \lambda_{b-k} =  -s_{m,n} + \frac{b-k}{2} > 0  >  s_{m,n,t}. 
\end{equation}
Keeping this in mind, the last Hom space in \eqref{E:kudla} can be
computed as
\begin{align}  \label{E:kudla}
 & \Hom_{G_{n}\times\GL(Y_b)\times H_{m-2b}}(R_{P(Y_b)}(\omega_{W_n,
   V_m}),  \pi \otimes  \chi_W
 |-|^{s_{m,n,t}} \cdot 1^{\times b} \otimes \Sigma_0 ) \notag \\
 \hookrightarrow&   \Hom_{G_{n}\times\GL(Y_b)\times H_{m-2b}}( J^b,
 \pi \otimes  \chi_W |-|^{s_{m,n,t}}  \cdot 1^{\times b} \otimes \Sigma_0) \notag \\
 =& \Hom_{\GL(X_b)\times G_{n-2b}\times H_{m-2b}}(
 \chi_V|-|^{-s_{m,n,t}} \cdot 1^{\times b} \otimes
 \omega_{W_{n-2b}, V_{m-2b}}, R_{\overline{Q}(X_b)}(\pi) \otimes
 \Sigma_0)\notag \\
=&\Hom_{\GL(X_b)\times G_{n-2b}}(
 \chi_V|-|^{-s_{m,n,t}} \cdot 1^{\times b} \otimes
 \Theta_{W_{n-2b}, V_{m-2b}}(\Sigma_0), R_{\overline{Q}(X_b)}(\pi))\notag\\
=&\Hom_{G_n}(
 \chi_V|-|^{-s_{m,n,t}} \cdot 1^{\times b} \rtimes
 \Theta_{W_{n-2b}, V_{m-2b}}(\Sigma_0), \pi )\notag
 \end{align}
where to obtain the second injection, we used the genericity of
$1^{\times b}$ and
\eqref{E:lambda} as before. Then again by dualizing and applying MVW,
we have
 \[  \pi  \hookrightarrow  \chi_V |-|^{s_{m,n,t}} \cdot 1^{\times b}  \rtimes \sigma_0 \]
 for some $\sigma_0$ which is a subquotient of $\Theta_{W_{n-2b},
   V_{m-2b}}(\Sigma_0)$. By the maximality of $a$, we conclude that
 $b\leq a$ and hence $b =a$.   This completes the proof of (i).

 \vskip 10pt
 
 \noindent (ii)  Suppose that $\Pi$ is given as in \eqref{E:Pi} with $a$ maximal and some $\Sigma$.
 Now that we know that $b=a$ in the proof of (i), we revisit  the computations starting from (\ref{E:kudla}):
\begin{align*}
 0 &\ne \Hom_{G_n\times H_m}\big( \omega_{W_n, V_m},  \pi \otimes  \Pi\big)\\
&\hookrightarrow  \Hom_{\GL(X_b)\times G_{n-2b}\times H_{m-2b}}(
 \chi_V |-|^{-s_{m,n,t}} \cdot 1^{\times b} \otimes
 \omega_{W_{n-2b}, V_{m-2b}}, R_{\overline{Q}(X_b)}(\pi) \otimes \Sigma) \\
&\hookrightarrow \Hom_{\GL(X_a)\times G_{n-2a}\times  H_{m-2a}}
\big(\chi_V |-|^{-s_{m,n,t}} \cdot 1^{\times a} \otimes \omega_{W_{n-2a}, V_{m-2a}},\\
&\hspace{2.5in} R_{\overline{Q}(X_a)} (\chi_V |-|^{s_{m,n,t}}\cdot 1^{\times a} \rtimes \sigma)\otimes
\Sigma\big).
\end{align*}
To show the proposition, it suffices to show that the last Hom space embeds into
\[
 \Hom_{G_{n-2a}\times H_{m-2a}}(\omega_{W_{n-2a}, V_{m-2a}} , \sigma \otimes \Sigma).
\]
To show this inclusion, we shall make use of Lemma \ref{L:geometric}. In Lemma \ref{L:geometric}(i),
 set $\sigma_{\rho, a}=\sigma_a$, namely set
$\rho=\chi_V |-|^{s_{m,n,t}}$. Tensoring the
 short exact sequence with $\Sigma$ and then applying the functor
 \[  
\Hom_{\GL(X_a)\times G_{n-2a}\times  H_{m-2a}}
(\chi_V \cdot |-|^{-s_{m,n,t}} \cdot 1^{\times a} \otimes\omega_{W_{n-2a}, V_{m-2a}}, -), 
\]
one sees that the desired inclusion follows from the assertion:
 \[
 \Hom_{\GL(X_a)\times G_{n-2a}\times  H_{m-2a}}
(\chi_V \cdot |-|^{-s_{m,n,t}} \cdot 1^{\times a} \otimes
 \omega_{W_{n-2a}, V_{m-2a}}, T \otimes \Sigma) = 0.
\]
But this follows from Lemma \ref{L:geometric}(ii) which asserts that $T$ does not
 contain any irreducible subquotient of the form
\[   \chi_W |-|^{-s_{m,n,t}} \cdot 1^{\times a} \otimes   \Sigma' \quad \text{for any $\Sigma'$.} \]
This completes the proof of (ii).
\end{proof}

 \vskip 15pt

\section{\bf Proof of Theorem \ref{T:main}}  \label{S:assem}

We can now assemble the results of the last two sections and complete
our proof of Theorem \ref{T:main}. As is mentioned in Section
\ref{S:outline},  we may assume that $m \leq n+ \epsilon_0$, or equivalently that $s_{m,n} \leq 0$ (see
\eqref{E:assumption}  and \eqref{E:s_m_n}),  because otherwise we may switch the roles of $G(W)$
and $H(V)$. Further, we shall argue by induction on $\dim W$. Thus, 
by induction hypothesis, we assume that Theorem \ref{T:main} is known for
dual pairs $G(W')  \times H(V')$ with $\dim V'  \leq \dim W'
+\epsilon_0 < n+\epsilon_0$.  \vskip 5pt

 \subsection{\bf Irreducibility.}
 We first show the irreducibility of $\theta(\pi)$.  If $\pi$ does not occur on the boundary of $I(-s_{m,n})$, this follows from Theorem \ref{T:non_boundary}.
 Thus, we assume $\pi$ occurs on the boundary, so that
\[  \pi \hookrightarrow \chi_V |{\det}_{\GL(X_t)}|^{s_{m,n} - \frac{t}{2}} \rtimes \sigma \]
for some $t >0$ and some $\sigma$. 
  Let us write
\[  \pi \hookrightarrow    \chi_V | -|^{s_{m,n,t}} \cdot 1^{\times a}   \rtimes \sigma  \]
with $a$ maximal. Let
$\Pi \subseteq\theta(\pi)$ be an irreducible submodule.  But by Proposition \ref{P:unique2}, 
\[   
0  \ne \Hom_{G_n\times H_m}( \omega_{W_n, V_m},  \pi \otimes \Pi)
\hookrightarrow \Hom_{G_{n-2a}\times H_{m-2a}}
(\omega_{W_{n-2a}, V_{m-2a}} , \sigma \otimes\Sigma),
\]
where $\Sigma \subset \theta_{W_{n-2a}, V_{m-2a}}(\sigma)$ is an irreducible representation such that
\[    \Pi \hookrightarrow   \chi_W |-|^{s_{m,n,t}} \cdot 1^{\times a}
\rtimes \Sigma  \]
where $a$ is also maximal for $\Pi$. By induction hypothesis, we have
$\Sigma=\theta_{W_{n-2a}, V_{m-2a}}(\sigma)$, and hence
\[    \Pi \hookrightarrow   \chi_W |-|^{s_{m,n,t}} \cdot 1^{\times a}
\rtimes \theta_{W_{n-2a}, V_{m-2a}}(\sigma). \]
 By Proposition \ref{P:unique1}, this induced representation has a
unique submodule. This shows that $\theta(\pi)$ is an isotypic representation. Further,
since
\[
\dim \Hom_{G_{n-2a}\times H_{m-2a}}
(\omega_{W_{n-2a}, V_{m-2a}} , \sigma \otimes\Sigma)=1,
\]
by the induction hypothesis, we conclude by Proposition \ref{P:unique2}(ii)
 that
\[
\dim \Hom_{G_n\times H_m}( \omega_{W_n, V_m},  \pi \otimes \Pi)=1.
\]
This shows that $\Pi$ occurs with multiplicity one in $\theta(\pi)$, so that  $\theta(\pi)$ is irreducible.
\vskip 10pt

\subsection{\bf Disjointness.}

It remains to prove that if $\theta(\pi)  = \theta(\pi') = \Pi \ne 0$, then $\pi = \pi'$. 
By Proposition \ref{P:non_boundary}, this holds unless both $\pi$ and
$\pi'$ occur on the boundary for the same $t$, namely there exists $t>0$ such that 
\[  \pi \hookrightarrow \chi_V |{\det}_{\GL(X_t)}|^{s_{m,n} - \frac{t}{2}} \rtimes \tau\]
and
\[  \pi' \hookrightarrow \chi_V |{\det}_{\GL(X_t)}|^{s_{m,n} - \frac{t}{2}} \rtimes \tau'. \]
This means that we may write
\[  \pi \hookrightarrow \chi_V |-|^{s_{m,n,t}} \cdot 1^{\times a} \rtimes \sigma \]
and 
\[  \pi' \hookrightarrow \chi_V |-|^{s_{m,n,t}} \cdot 1^{\times a'} \rtimes \sigma' \]
with $a$ and $a'$ maximal (and for some $\sigma$ and $\sigma'$).  But
then by Proposition \ref{P:unique2}(i) one
must have $a=a'$, where $a$ is maximal such that
\[  \Pi  \hookrightarrow  \chi_W |-|^{s_{m,n,t}} \cdot 1^{\times a}
\rtimes \Sigma \quad \text{for some $\Sigma$.} \]
Moreover, with Proposition \ref{P:unique2}(ii) and the induction hypothesis, we have
\[  \theta_{W_{n-2a}, V_{m-2a}}(\sigma)  =\Sigma  =  \theta_{W_{n-2a}, V_{m-2a}}(\sigma'),\]
so  that $\sigma  \cong \sigma'$. We then deduce by
Proposition \ref{P:unique1} that $\pi \cong \pi'$ is the unique
irreducible submodule of
$ \sigma_a = \chi_V |-|^{s_{m,n,t}} \cdot 1^{\times a} \rtimes \sigma$. This
completes the proof of Theorem \ref{T:main}.

 \vskip 15pt

\section{\bf Quaternionic Dual Pairs }

In this final section, we consider the case of the quaternionic
dual pairs. As mentioned in the introduction, due to the lack of the
MVW involution, we are not able to prove the Howe duality conjecture
in full generality. The best we can prove is Theorem
\ref{T:main_quaternion}, where we consider only Hermitian
representations (i.e.\ those $\pi$ such that $\overline{\pi} \cong
\pi^{\vee}$, where $\overline{\pi}$ is the complex conjugate of
$\pi$).  The idea of the proof is essentially the same as for the
non-quaternionic
case. But in place of the MVW involution, we use the involution
\[
\pi\mapsto \overline{\pi}.
\]
In what follows, we will outline how to modify the proof. 
\vskip 5pt

\subsection{\bf Setup.}
Let us briefly recall the setup in the quaternionic case, with
emphasis on the  aspects which are different from before.
\vskip 5pt

Let $B$ be the unique quaternion division algebra over $F$. For $\epsilon = \pm$,
let $W=W_n$ be a rank $n$ $B$-module equipped with a $-\epsilon$-Hermitian form and  $V_m$  a rank $m$ $B$-module equipped with an $\epsilon$-Hermitian form. 
Then the product $G(W_n) \times H(V_m)$ of isometry groups is a dual pair,  with a  Weil representation $\omega_{W,V}$ associated to a pair of splitting characters $(\chi_V, \chi_W)$. 
In this case, the characters $\chi_V$ and $\chi_W$ are simply (possibly trivial) quadratic characters determined by the discriminants of the corresponding spaces $V$ and $W$.    
\vskip 5pt

For an isotropic subspace $X_t$ of rank $t$ over $B$, let $Q(X_t)$ be
the stabilizer of $X_t$, which is a maximal parabolic subgroup of
$G(W)$ with the Levi factor $L(X_t)  \cong \GL(X_t) \times
G(W_{n-2t})$, where $\GL(X_t)  \cong \GL_t(B)$.  We shall denote by
${\det}_{\GL(X_t)}:  \GL(X_t)  \longrightarrow F^{\times}$ the reduced
norm map.  Likewise, a maximal parabolic subgroup $P(Y_t)$ of $H(V_m)$
is the stabilizer of an isotropic subspace $Y_t$ of $V_m$.  As before,
we have Tadi\'{c}'s notation for parabolic induction.  
\vskip 5pt

We set
\[  s_{m,n}  = m-n + \frac{\epsilon}{2}. \]
In the quaternionic case, we say that an irreducible representation $\pi$ of $G(W_n)$ lies on the boundary of $I(-s_{m,n})$ if there exists $0< t \leq q_W $ such that
\[  \pi \hookrightarrow \chi_V |{\det}_{\GL(X_t)}|^{s_{m,n} - t}  \rtimes \sigma \]
for some irreducible representation $\sigma$ of $G(W_{n-2t})$.  If
$\overline{\pi} \cong \pi^{\vee}$, i.e.\ if $\pi$ is Hermitian, then
by dualizing and complex conjugating, we see that this is 
equivalent to:
\[  \Hom_{\GL(X_t)}(\chi_V | {\det}_{\GL(X_t)}|^{-s_{m,n} +t}, R_{\overline{Q}(X_t)}(\pi))  \ne  0. \]

\vskip 5pt

\subsection{\bf Non-boundary case.}
We can now begin the proof of Theorem \ref{T:main_quaternion}, starting with the case when the Hermitian representation $\pi$ does not lie on the boundary. 
For the sake of proving Theorem \ref{T:main_quaternion}, there is no loss of generality in assuming that $m < n - \frac{\epsilon}{2}$, so that $s_{m,n}  < 0$.  
\vskip 5pt

Now one can verify that all the arguments in Section
\ref{S:non_boundary_case} continue to work for a Hermitian $\pi$, with the following modifications:
\vskip 5pt

\begin{enumerate}[$\bullet$]
\item  The first  place where the MVW
involution is used in Section \ref{S:non_boundary_case} is the see-saw identity
\eqref{E:see-saw}. But one can see from the proof of the see-saw
identity in \cite[\S 6.1]{gi} that one has 
\begin{equation*}
\Hom_{G(W) \times G(W)} ( \Theta_{V, W + W^-}(\chi_W), \pi' \otimes
\pi^{\vee}) = \Hom_{H(V)^\Delta}( \Theta(\pi') \otimes
\overline{\Theta(\pi)}, \CC). 
\end{equation*}
Then \eqref{E:see-saw2} can be written as
\begin{align}
 \Hom_{H(V)^\Delta}( \Theta(\pi') \otimes
\overline{\Theta(\pi)}, \CC) 
&\supseteq  \Hom_{H(V)^\Delta}( \theta_{her}(\pi') \otimes
\theta_{her}(\pi)^{\vee}, \CC)  \notag \\
&= \Hom_{H(V)}( \theta_{her}(\pi'),  \theta_{her}(\pi)).  \notag
\end{align}
 
\vskip 5pt

\item There is an evident analogue of Lemma \ref{L:key} in the
  quaternionic case. The statement is as given in Lemma \ref{L:key},
  except that the terms $|\det_{\GL(X_t)}|^{s+ \frac{t}{2}}$
should be replaced by $|\det_{\GL(X_t)}|^{s+ t}$.
\end{enumerate}
With these provisions, the rest of the argument
in Section \ref{S:non_boundary_case} does not use the MVW involution,
and hence apply to the
quaternionic case without any modification. One also has the analogue of Proposition  \ref{P:non_boundary}, with the exponent $s_{m,n}  - \frac{t}{2}$ replaced by $s_{m,n}  -t$.

\vskip 5pt

\subsection{\bf Boundary case.}
Suppose that $\pi$ is an irreducible Hermitian representation of
$G(W)$ which lies on the boundary of $I(-s_{m,n})$.
Then, for some $t > 0$, one has
\[  \pi \hookrightarrow    \chi_V | -|^{s_{m,n,t}} \cdot 1^{\times a}   \rtimes \sigma  \]
with $a$ maximal (and for some $\sigma$) and
\[  s_{m,n, t}  =  s_{m,n}  -2t +1  < 0. \]
Now  one has an analogue of Lemma \ref{L:geometric}, based on  the
explicit Geometric Lemma in the quaternionic case (which is written
down in the thesis of M. Hanzer \cite[Theorem 2.2.5]{Ha}).    Using
this, one deduces the analogue of Proposition \ref{P:unique1} by the
same argument. However, it is essential to note the following lemma:

\vskip 5pt

\begin{lem}  \label{L:her}
Suppose that
\[  \pi  \hookrightarrow      \chi_V | -|^{s_{m,n,t}} \cdot 1^{\times a}   \rtimes \sigma  \]
with $a$ maximal and for some $\sigma$. If $\pi$ is Hermitian, so is $\sigma$.
\end{lem}

\begin{proof}
To see
this, starting from $\pi \hookrightarrow  \rho^{\times a} \rtimes
\sigma$ (with $a$ maximal and $\rho$ a real-valued 1-dimensional
character for which $\rho^{\vee} \ne \rho$), one deduces (by dualizing
and complex-conjugating) that
\[  (\overline{\rho}^{\vee})^{\times a} \rtimes
\overline{\sigma}^{\vee} \twoheadrightarrow \overline{\pi}^{\vee}
\cong \pi. \]
Now note that for the case at hand, the supercuspidal representation
$\rho$ satisfies $\overline{\rho}  = \rho$; indeed, $\rho$ is a
real-valued character for our application. Thus, Bernstein's
Frobenius reciprocity implies that 
\[  (\rho^{\vee})^{\times a} \otimes \overline{\sigma}^{\vee}
\hookrightarrow R_{\overline{Q}(X_{ra})}( \pi). \]
However, the analogue of Lemma \ref{L:geometric}(i) says that the only
irreducible subquotient of 
$R_{\overline{Q}(X_{ra})}( \pi)$ of the form $(\rho^{\vee})^{\times a} \otimes \sigma_0$ is 
 $(\rho^{\vee})^{\times a} \otimes \sigma$. Hence, we see that
 $\overline{\sigma}^{\vee}  \cong \sigma$.
\end{proof}
\vskip 5pt

Then one has the following analogue of Proposition \ref{P:unique2}:
\vskip 5pt
\begin{prop}\label{P:unique2-quat}
Assume that $\pi$ is an irreducible Hermitian representation of $G(W)$
and  $0 \ne \Pi\subset \theta_{her}(\pi)$.
\vskip 5pt

\noindent (i)  If
\[  \pi \hookrightarrow    \chi_V | -|^{s_{m,n,t}} \cdot 1^{\times a}   \rtimes \sigma  \]
with $a$ maximal (and for some $\sigma$, necessarily Hermitian by Lemma \ref{L:her}), then
\[    \Pi \hookrightarrow   \chi_W |-|^{s_{m,n,t}} \cdot 1^{\times a} \rtimes \Sigma  \]
for some $\Sigma$ (necessarily Hermitian by Lemma \ref{L:her}) and  where $a$ is also maximal for $\Pi$. 

\vskip 5pt
\noindent (ii) Moreover, whenever $\Pi$ is presented as a submodule as above, one has
\[   
0  \ne \Hom_{G_n\times H_m}( \omega_{W_n, V_m},  \pi \otimes \Pi)
\hookrightarrow \Hom_{G_{n-2a}\times H_{m-2a}}
(\omega_{W_{n-2a}, V_{m-2a}} , \sigma \otimes\Sigma),
\]
 so  that $\Sigma\subseteq \theta_{W_{n-2a}, V_{m-2a}, her}(\sigma)$.
\end{prop}

Proposition \ref{P:unique2-quat} is proved by the same argument as that for
Proposition \ref{P:unique2}, using the analogue of Lemma \ref{L:kudla} (see \cite[Chap. 3, Sect. IV, Thm. 5, Pg. 70]{mvw}).  We only take note that in the statement  of 
Lemma \ref{L:kudla}, the quantity $\lambda_{a-k}$ should be equal to $s_{m,n}  + a-k$ in the quaternionic case.
\vskip 10pt

With the above provisions, the rest of the proof goes through for the quaternionic case, which completes the proof of
 Theorem \ref{T:main_quaternion}.
\vskip 15pt

\section*{\bf Appendix: Proof of Lemma \ref{L:geometric}}
The goal of this appendix is to prove the technical Lemma
\ref{L:geometric}. We restate the lemma here for the convenience of
the reader.
\vskip 5pt

\begin{lem_geometric}
Let $\rho$ be a supercuspidal representation of $\GL_r(E)$ and
consider the induced representation 
\[
\sigma_{\rho,a} = \rho^{\times a}  \rtimes \sigma
\] 
of $G(W_n)$ where $\sigma$ is an irreducible  representation  of
$G(W_{n-ra})$. Assume that
\begin{itemize}
\item[(a)]  $^c\rho^{\vee}  \ne \rho$;
\item[(b)]  $ \sigma \nsubseteq \rho \rtimes \sigma_0$ for any
  $\sigma_0$.
\end{itemize}
 Then we have the following:
\begin{enumerate}[(i)]
\item One has a natural short exact sequence
\[  
\begin{CD}
0 @>>> T @>>>  R_{\overline{Q}(X_{ra})} \sigma_{\rho, a} @>>>
({^c\rho}^{\vee})^{\times a}  \otimes \sigma @>>> 0 \end{CD} 
\]
and $T$ does not  contain any irreducible subquotient of the form
$({^c\rho}^{\vee})^{\times a} \rtimes \sigma'$ for any $\sigma'$. 
In particular, $R_{\overline{Q}(X_{ra})} \sigma_{\rho, a}$
contains $({^c\rho}^{\vee})^{\times a}  \otimes \sigma$ with multiplicity one, and
does not contain any other subquotient of the form $({^c\rho}^{\vee})^{\times a}  \otimes \sigma'$.
Likewise, $R_{Q(X_{ra})} \sigma_{\rho, a}$
contains $\rho^{\times a}  \otimes \sigma$ with multiplicity one and does not contain any other subquotient of the form
$\rho^{\times a}  \otimes \sigma'$. 

\vskip 5pt

\item  The induced representation $\sigma_{\rho, a}$
has a unique irreducible submodule.
\end{enumerate}
\end{lem_geometric}

\begin{proof}
 We shall use an explication of the Geometric Lemma of
 Bernstein-Zelevinsky due to Tadi\'c \cite[Lemmas 5.1 and
 6.3]{ta}. (See \cite{hm} for the metaplectic group.) Tadi\'c's
 results imply that any irreducible subquotient $\delta \otimes
 \sigma'$ of $ R_{\overline{Q}(X_{ra})} \sigma_{\rho, a}$  is obtained in the following way.  
\vskip 5pt

For any partition $k_1 + k_2 + k_3=ra$, write the semisimplification
of the normalized Jacquet module of $\rho^{\times a}$ to the Levi
subgroup $ \GL_{k_1} \times \GL_{k_2} \times \GL_{k_3}$ as a sum of
$\delta_1 \otimes \delta_2 \otimes \delta_3$.   Similarly, write the
semisimplification of the normalized Jacquet module of $\sigma$ to the
Levi subgroup $\GL_{k_2} \times G(W_{n-2ra-2k_2})$ as a sum of
$\delta_4 \otimes \sigma_5$.  Then 
$\delta$ is a subquotient of  $\delta_3 \times ^c\delta^{\vee}_1 \times
^c\delta_4^{\vee}$ whereas $\sigma'$ is a subquotient of $ \delta_2 \otimes
\sigma_5$.
\vskip 5pt

For the case at hand,  since $\rho$ is supercuspidal, we can assume the partition of
$ra$ is of the form $rk_1+rk_2+rk_3=ra$, and the (semisimplified)
normalized Jacquet module of $\rho^{\times a}$ is the isotypic sum of
$\rho^{\times k_1} \otimes \rho^{\times k_2} \otimes \rho^{\times
  k_3}$. Hence we see that for any irreducible subquotient $\delta
\otimes  \sigma'$ of $ R_{\overline{Q}(X_{ra})}\sigma_{\rho, a}$, $\delta$ is
a subquotient of   $\rho^{\times k_3} \times ({^c\rho}^{\times
  k_1})^{\vee}  \times ^c\delta^{\vee}_4$. 
\vskip 5pt

Now the irreducible subquotients of $T$ correspond to
those partitions with $k_2>0$ or $k_3>0$. (Note that the case $k_2 =
k_3  =0$ corresponds to the closed cell in $Q\backslash
G\slash\overline{Q}$, which gives the third term in the short exact
sequence.) The conditions (a) and
(b) then imply that $\delta \ne (^c \rho^{\vee})^{\times a}$. This
proves the statements about $T$ in (i).  Now $Q(X_{ra})$ and
$\overline{Q}(X_{ra})$ are conjugate in $G(W)$ by an element $w$ which
normalizes the Levi subgroup $L(X_{ra}) = \GL(X_{ra}) \times
G(W_{n-2ra})$, acting as  the  identity on $G(W_{n-2ra})$ and 
via $g \mapsto {^c({^t}g^{-1})}$ on $\GL(X_{ra})$. Then one has
\[
^wR_{\overline{Q}(X_{ra})}(\sigma_{\rho, a})=R_{Q(X_{ra})}(\sigma_{\rho, a})
\]
where the LHS is the representation of the Levi $L(X_{ra})$ obtained
by twisting $R_{\overline{Q}(X_{ra})}(\sigma_{\rho, a})$ by $w$.
Hence, one deduces that 
$R_{Q(X_{ra})}(\sigma_{\rho, a})$ contains $\rho^{\times a} \otimes \sigma$ with multiplicity one. 
\vskip 5pt

Finally, for (ii), let  $\pi\subseteq\sigma_{\rho, a}$ be any irreducible
submodule. Then the Frobenius reciprocity implies that
the semisimplification of $R_{Q(X_{ra})}(\pi)$ contains $\rho^{\times a}\otimes\sigma$. 
Thus, if $\sigma_{\rho, a}$ contains more than one irreducible submodule, the
exactness of the Jacquet functor implies that  
$R_{{Q}(X_{ra})}(\sigma_{\rho, a})$ contains $\rho^{\times a}\otimes\sigma$ with
multiplicity $\geq 2$, which contradicts (i).
 \end{proof}

 \vskip 10pt

\bibliographystyle{amsalpha}

\end{document}